\documentclass[11pt]{amsart}

\usepackage[left=1in, right=1in, top=1.2in, bottom=1.2in]{geometry}
\usepackage{microtype, mathtools, mathrsfs, enumerate, dsfont, esvect}
\usepackage{tikz,wasysym}
\usepackage{pgf}
\usepackage{amssymb}
\usetikzlibrary{positioning,automata}
\usetikzlibrary{arrows,automata,positioning}
\usepackage[font=small,labelfont=bf]{caption}
\usepackage{bbm}
\usepackage{verbatim}
\usepackage[linktocpage]{hyperref}
\hypersetup{
    colorlinks=true,        % false: boxed links; true: colored links
    linkcolor=red,          % color of internal links (change box color with linkbordercolor)
    citecolor=blue,         % color of links to bibliography
    filecolor=magenta,      % color of file links
    urlcolor=cyan           % color of external links
}

\newtheorem{thm}{Theorem}

\newtheorem{prop}[thm]{Proposition}

\newtheorem{lem}[thm]{Lemma}
\newtheorem{defn}[thm]{Definition}
\newtheorem{cor}[thm]{Corollary}
\newtheorem*{thm*}{Theorem}
\newtheorem*{prop*}{Proposition}
\newtheorem{rmk}[thm]{Remark}
\numberwithin{thm}{section}

\newcommand{\calR}{\mathcal{R}}

\newcommand{\calL}{\mathcal{L}}
\newcommand{\calP}{\mathcal{P}}
\newcommand{\calF}{\mathcal{F}}
\newcommand{\calM}{\mathcal{M}}

\newcommand{\N}{\mathbb{N}}

\newcommand{\Q}{\mathbb{Q}}

\newcommand{\I}{\mathbb{I}}

\newcommand{\Fin}{\mathcal F}
\newcommand{\W}{\mathcal W}
\newcommand{\Int}{\operatorname{Int}}

\title{O-minimal open core is not an elementary property}
\author{Alexi Block Gorman}
\address{Universiteit van Amsterdam\\
Institute for Logic, Language and Computation\\
Science Park 900\\ 
1098 XH Amsterdam\\
Netherlands} \email{a.t.blockgorman@uva.nl}
\author{Esther Elbaz Saban}

\begin{document}

\begin{abstract}
Given a structure $\calM$ with a definable topology, its open core is a structure defined on the same universe whose language consists of all open sets of all arities definable in $\calM$.
In response to questions raised by Dolich, Miller, and Steinhorn in their early work on open core, we prove that having an o-minimal open core is not an elementary property.
In particular, we construct an expansion of the structure $(\Q,<)$ that has an o-minimal open core, but some of its elementary superstructures do not.
\end{abstract}

\maketitle

\section{Introduction}

The notion of an open core has become a mainstay of the model-theoretic study of ordered structures. 
Introduced by Miller and Speissegger in \cite{MS99}, the open core of a structure is meant to capture a geometric notion of tameness that allows for some degree of ``noise'' or complexity, so long as such pathologies do not interact too much with the structure's natural topology.
It was established by Hieronymi, Nell and Walsberg in \cite{HNW17} that structures can be maximally complex in terms of neostability while still having tame (even o-minimal) open cores.

However, there are still many rather basic questions, in particular about structures with o-minimal open cores, that have remained unanswered. 
In this paper, we address one of the natural questions raised by the work of Dolich, Miller, and Steinhorn in \cite{DMS1, DMS2}.
Namely, we show that the property of having an o-minimal open core is not elementary.

Formally, the open core of a first order structure is defined as follows.
Given a structure $\mathcal{M}$ with a definable topology $\tau$, we let $\mathcal{M}^{\circ}$ denote the structure $(M,(U))$, where $U$ ranges over the open sets (with respect to $\tau$) of all arities definable (with parameters) in $\mathcal{M}$, and call this structure \textbf{the open core of $\mathcal{M}$}. 
We say that a structure $\calR:=(R,<,\ldots)$ has \textbf{o-minimal open core} if $\calR$ extends the theory DLO (dense linear orders without endpoints), and $\calR^{\circ}$ is o-minimal (i.e.\ every unary set definable with parameters is a finite union of points and intervals).

At this point, we should caution the reader against disregarding the ``with parameters'' caveat in the above definitions.
It is straightforward to cook up an example of a structure whose $\emptyset$-definable unary sets consist only of points and intervals, but whose definable sets with parameters do not share this property.
Hence, for the notion of ``o-minimal open core'' to share even a fraction of the richness of o-minimality, we must at all times understand ``definable'' to mean definable with parameters.

We sometimes state that some structure ``is'' the open core of another structure, despite its language \emph{not} consisting precisely of every open set in every arity that is definable in the original structure.
Given two structures $\mathcal{S}_1$ and $\mathcal{S}_2$ with the same universe $S$, we say $\mathcal{S}_1$ and $\mathcal{S}_2$ are \textbf{interdefinable} if $\mathcal{S}_1$ and $\mathcal{S}_2$ define the same sets. 
For a theory $T$ extending the theory of dense linear orders, we say that a theory $T'$ is an \textbf{open core of $T$} if for every $\mathcal{M} \models T$ there exists $\mathcal{M}' \models T'$ such that $\mathcal{M}^{\circ}$ and $ \mathcal{M}'$ are interdefinable.

There has been a recent trend (see, e.g., \cite{BHK20}, \cite{Z26}) of trying to identify 
the open core of a (complete) theory in the above sense, rather than just structures.
However, the primary result of this paper illustrates the limitations on such results.
Our main result is the following.

\begin{thm}\label{Main}
The property of having o-minimal open core is not elementary.
\end{thm}

We provide an explicit example to demonstrate this claim.
As one might expect, this is most easily achieved by avoiding any sort of algebraic structure whatsoever.
Our example expands the rational numbers $\Q$ by a binary relation chosen to be as conservative an expansion as possible while still providing a counterexample.
Such examples surely exist with additional algebraic structure, one would likely just need to define the special relation that violates uniform finiteness in a such a way that the fibers avoid all algebraic subsets.

It is natural to wonder why no such example has appeared ``in the wild'' before now, and why we should need to construct such an artificial example.
The answer lies in the prevalence of uniform finiteness in structures whose open cores have been considered.
We say that a structure is \textbf{uniformly finite} if for every definable set $D$ of arity $n$, if the one-dimensional fibers of $D$ are all finite, then there is some $N \in \N$ that bounds the size of all of $D$'s one-dimensional fibers.
Uniform finiteness plays an essential role in many important theorems within the realm of tame geometry, such as the o-minimal cell decomposition theorem.

Uniform finiteness is also used crucially in the work of Boxall and Hieronymi in \cite{BH12} on properties of open cores.
Our main result was inspired by the work of Boxall and Hieronymi in \cite{BH12}, though we do not use their results directly.
For researchers in tame geometry (in particular, the study of topological tameness beyond o-minimality), it has been a solid rule of thumb that for a structure to have any hope of being ``tame,'' it must be uniformly finite.
This paper serves in large part to justify this intuition.

The example we construct can be interpreted in the \textbf{weak monadic second order theory} of a dense linear order with only a left endpoint.
We write $\W(\I)$ to denote this theory, where $\I$ is the dense linear order in question.
Monadic second order theories have long been a prominent object of study in their own right, and have been a valuable source of tools and examples in various model theory- and computability theory-related contexts.
Throughout this paper, we will roughly follow the terminology and notation of Linkhorn in \cite{L21}, to which we refer the reader for more context and exposition about the monadic second order theory of DLO as well.

\section{Preliminaries}
\subsection{The weak monadic second order theory of DLO}\label{WSO}

We first define the weak monadic second order theory of a dense linear order with a left endpoint, but no right endpoint.
In \cite{L21}, the author proves that the weak monadic second order theory of this structure is model complete in the following language:
$$\calL_{W(\I )} =\{\cup ,\cap ,\emptyset ,0,\min,\max,s^{-1}\}.$$

We write $W(\I )$ to denote the structure in this language such that:
\begin{itemize}
\item The universe $\calP_{fin}(\I)$ is the collection of finite subsets of $\I$;
\item $\cup, \cap$ are interpreted as the usual union and intersection;
\item $\emptyset$ is interpreted as the empty set;
\item $0$ is interpreted as the least element of $\I$;
\item $\min$ and $\max$ are interpreted as the operations taking a non-empty finite set to the singleton containing its minimum and maximum respectively, with respect to the ordering of $\I$ (and both fixing $\emptyset$);
\item $s^{-1}$ is the binary function given by $s^{-1}(A,B) = \{i \in A : s_A(i) \in B\}$,
i.e. $s^{-1}$ is a single binary function which encodes the (preimage map associated to) the family of successor functions on finite subsets of $I$.
\end{itemize}

Furthermore, let $i: \I \to \calP_{fin}(\I)$ be the map $x \mapsto \{x\}$ identifying elements of $\I$ with the corresponding singleton in $\calP_{fin}(\I)$.
It is easy to see that $(\I,<)$ can be interpreted in $W(\I)$, and moreover that the induced structure on the interpretation of $(\I,<)$ on the set of singletons is just that of $(\I,<)$.

Throughout this paper, we will work with the analogous weak monadic second order theory $W(\Q_{>0})$ in which the dense linear order has no endpoints.
The presence of a left endpoint, or in our case lack thereof, will make no material difference on account of the specific results we will apply from this context.

\subsection{Construction of $(\Q,<,0,A)$}

Our setting is the dense linear order of rational numbers $(\Q,<)$.
Our aim is to construct a set $A\subseteq \Q^2$ such that:
\begin{itemize}
    \item[(P1)] For every $x\in \Q$, the fiber $A_x:=\{y: (x,y) \in A\}$ is finite;
    \item[(P2)] For every $N\in \N$, there exists $x\in \Q$ such that $A_x$ has cardinality greater than $N$;
    \item[(P3)] $(\Q, <, A)$ has an o-minimal open core.
\end{itemize}

Let $\Fin:=[\mathbb{Q}_{>0}]^{<\omega}$
denote the set of finite subsets of~\(\mathbb{Q}_{>0}\).

\begin{defn}[Dense coding]\label{lem:dense-coding}
Fix a map
$\rho:\mathbb{Q}_{<0}\to \Fin $
such that, for every \(F\in\Fin\), the set
$D_F:=\rho^{-1}(F)$
is dense and codense in \(\mathbb{Q}_{<0}\).
\end{defn}

We define $\calM=(\mathbb{Q},<,0,A)$ via
\[ A(x,y)
\quad\Longleftrightarrow\quad
x<0<y\text{ and }y\in\rho(x). \]
Equivalently, for \(x<0\) we define the fiber of $x$ in $A$:
\[
A_x:=\{y\in\mathbb{Q}:M\models A(x,y) \} =\rho(x),
\]
and for \(x\geq 0\), we set
$A_x=\emptyset$.

In $(\mathbb{Q},<,0,A)$, every fiber \(A_x\) is finite, and the sizes of the fibers \(A_x\) are not uniformly bounded.
It is clear that $A$ satisfies (P1) and (P2).

We will express the weak monadic structure associated with 
\((\mathbb{Q}_{>0},<)\) as follows:
\[ \W(\mathbb{Q}_{>0}):=(\mathbb{Q}_{>0},\Fin,<,\in), \]
where \(\in\subseteq \mathbb{Q}_{>0}\times\Fin\) is membership.
We write $\calL_{\W(\Q_{>0})} = \{\Fin,<,\in \}$ to denote the language of this structure.

\section{Interpreting $(\Q, <,0, A)$ in $\W(\Q_{\geq 0})$}

To interpret our structure $(\Q, <,0, A)$ in $\W(\Q_{\geq 0})$ in such a way as to be useful, we must first establish some terminology and definitions, mostly following the standard conventions from the literature on o-minimality.

\begin{defn}
Let \(C\subseteq \mathbb{Q}_{<0}\) be finite, and let
$\bar x=(x_1,\dots,x_r)$
be variables ranging over \(\mathbb{Q}_{<0}\).

A \emph{complete order cell over \(C\)} in the variables \(\bar x\) is a maximal consistent conjunction deciding all order relations among the variables \(x_i\) and the parameters from \(C\). More explicitly, it is a conjunction which, for every \(i,j \leq r\), contains exactly one of the following:
\[
x_i<x_j,\qquad x_i=x_j,\qquad x_j<x_i,
\]
and, for every \(i\leq r\) and every \(c\in C\), contains exactly one of the following:
\[
x_i<c,\qquad x_i=c,\qquad c<x_i.
\]

We identify a complete order cell with the subset of \(\bigl(\mathbb{Q}_{<0}\bigr)^r\) that it defines. The complete order cells over \(C\) form a finite partition of \(\bigl(\mathbb{Q}_{<0}\bigr)^r\).
\end{defn}

Let
$ \varphi(\bar x,\bar y;\bar a)$
be an \(\calL_{\W(\mathbb{Q}_{>0})}\)-formula whose free variables range as follows:
\[
\bar x=(x_1,\dots,x_r)\in \bigl(\mathbb{Q}_{<0}\bigr)^r,
\qquad
\bar y=(y_1,\dots,y_s)\in \bigl(\mathbb{Q}_{>0}\bigr)^s.
\]
We may partition the parameters \(\bar a\) into $\bar{c} \cup \bar{d}$ consisting of negative parameters
$\bar c=(c_1,\dots,c_m)\in \bigl(\mathbb{Q}_{<0}\bigr)^m,$
positive parameters
$\bar d\in \bigl(\mathbb{Q}_{>0}\bigr)^\ell$ (and we omit the parameter $0$, as it is a constant in our language).

\begin{prop}
\label{prop:relative-reduction}
On \(\bigl(\mathbb{Q}_{<0}\bigr)^r\times \bigl(\mathbb{Q}_{>0}\bigr)^s\), the formula \(\varphi(\bar x,\bar y;\bar a)\) is equivalent to a finite disjunction of formulas of the form
\[
\chi_i(\bar x;\bar c)
\wedge
\Theta_i\bigl(
\rho(x_1),\dots,\rho(x_r),\bar y;
\rho(c_1),\dots,\rho(c_m),\bar d
\bigr),
\]
where:
\begin{itemize}
    \item \(\chi_i(\bar x;\bar c)\) is a $\{<\}$-formula defining a subset of \(\mathbb{Q}_{<0}^r\) with parameters \(\bar c\);
    \item \(\Theta_i\) is a $\calL_{W(\Q_{>0})}$-formula 
    %where $x_1, \ldots ,x_r$ range over $\Q_{<0}$, 
    where $\rho(x_1),\dots,\rho(x_r)\in \calF$ and \(\bar y\) ranges over $\Q_{>0}$.
\end{itemize}
\end{prop}

Note that for an arbitrary formula with variables ranging over \(\mathbb{Q}\), the same conclusion holds after partitioning the variables according to the signs
$z_i<0$, $ z_i=0$,  and $ z_i>0$ for all variables.

\emph{Remark.}
The notation
\[
\Theta(\rho(\bar x),\bar y)
\]
is external notation. Internally, in $(\Q,<,0,A)$, it is read through the quotient of \(\mathbb{Q}_{<0}\) by the definable equivalence relation
\[
x\sim x'
\quad\Longleftrightarrow\quad
\forall y>0\bigl(A(x,y)\leftrightarrow A(x',y)\bigr).
\]
Thus we have:
\[
x\sim x'\quad\Longleftrightarrow\quad \rho(x)=\rho(x').
\]
Hence \(\mathbb{Q}_{<0}/{\sim}\) is naturally identified with \([\mathbb{Q}_{>0}]^{<\omega}\). Under this identification, the membership relation
$y\in\rho(x)$
is exactly
$A(x,y)$.
Thus formulas of \(\W(\mathbb{Q}_{>0})\) can be read as formulas in $(\Q,<,0,A)$, using negative representatives for finite subsets of \(\mathbb{Q}_{>0}\).

\begin{proof}
We prove the statement by induction on formulas.

It will be convenient for us to work in the following auxiliary three-sorted structure:
\[
\mathcal C=(\mathbb{Q}_{<0},<;\ \mathbb{Q}_{>0},<;\ \Fin,\in;\ \rho:\mathbb{Q}_{<0}\to\Fin).
\]
In this structure, the relation \(A\) is translated as follows:
\[
A(x,y)\quad\Longleftrightarrow\quad x\in \mathbb{Q}_{<0},\ y\in \mathbb{Q}_{>0},\ y\in\rho(x).
\]

We say that a formula in variables
\[
\bar x\in \bigl(\mathbb{Q}_{<0}\bigr)^r,\qquad \bar y\in \bigl(\mathbb{Q}_{>0}\bigr)^s
\]
is in \textbf{relative normal form} if it is a finite disjunction of formulas of the following form:
\[
\chi(\bar x)\wedge \Theta(\rho(\bar x),\bar y)
\]
where \(\chi\) is a formula of the pure order on \(\mathbb{Q}_{<0}\), and \(\Theta\) is a formula of \(\W(\mathbb{Q}_{>0})\). 
Parameters are allowed as in the statement of the proposition.

We show that every formula is equivalent, after fixing the signs of the variables, to a formula in relative normal form.

\noindent\textbf{Atomic formulas.}
The atomic formulas are of the form
$u=v$,$u<v$, and $A(u,v)$.
If \(u,v\) are both negative, then \(u=v\) and \(u<v\) are pure order formulas on \(\mathbb{Q}_{<0}\), hence belong to the \(\chi\)-part.

If \(u,v\) are both positive, then \(u=v\) and \(u<v\) are formulas in the singleton sort of \(\W(\mathbb{Q}_{>0})\), hence belong to the \(\Theta\)-part.

If one variable is negative and the other positive, then the truth of \(u=v\) and \(u<v\) is determined by the signs, and are partitioned across $\chi$ and $\theta$ as such.

Finally, \(A(u,v)\) is false unless \(u<0<v\). On \(\mathbb{Q}_{<0}\times \mathbb{Q}_{>0}\), we have
\[
A(u,v)\quad\Longleftrightarrow\quad v\in\rho(u),
\]
which is a formula of \(\W(\mathbb{Q}_{>0})\), with variable \(\rho(u)\) ranging over $\Fin$ and variable \(v\) ranging over the singleton sort. Thus every atomic formula is in relative normal form.

\medskip

\noindent\textbf{Boolean combinations.}
The class of formulas in relative normal form is closed under finite disjunctions.

It is also closed under conjunctions, because
\[
\bigl(\chi_1(\bar x)\wedge\Theta_1(\rho(\bar x),\bar y)\bigr)
\wedge
\bigl(\chi_2(\bar x)\wedge\Theta_2(\rho(\bar x),\bar y)\bigr)
\]
is equivalent to
\[
\bigl(\chi_1(\bar x)\wedge\chi_2(\bar x)\bigr)
\wedge
\bigl(\Theta_1(\rho(\bar x),\bar y)\wedge\Theta_2(\rho(\bar x),\bar y)\bigr).
\]

It is closed under negation as well. Indeed, after distributing over finite disjunctions, negation reduces to formulas of the form
\[
\neg\bigl(\chi(\bar x)\wedge\Theta(\rho(\bar x),\bar y)\bigr),
\]
which is equivalent to
\[
\neg\chi(\bar x)\vee\neg\Theta(\rho(\bar x),\bar y).
\]
This is again a finite disjunction of formulas in relative normal form.

\medskip

\noindent\textbf{Quantification over a positive variable.}
Suppose first that we have
\[
\exists z>0\ \psi(z,\bar x,\bar y).
\]
By the induction hypothesis, \(\psi(z,\bar x,\bar y)\) is equivalent to a finite disjunction of formulas
\[
\chi_i(\bar x)\wedge\Theta_i(\rho(\bar x),z,\bar y).
\]
Therefore
\[
\exists z>0\ \psi(z,\bar x,\bar y)
\]
is equivalent to
\[
\bigvee_i
\left(
\chi_i(\bar x)
\wedge
\exists z\in \mathbb{Q}_{>0}\ \Theta_i(\rho(\bar x),z,\bar y)
\right).
\]
The formula
\[
\exists z\in \mathbb{Q}_{>0}\ \Theta_i(\rho(\bar x),z,\bar y)
\]
is still a formula of \(\W(\mathbb{Q}_{>0})\). Hence the resulting formula is in relative normal form.

\medskip

\noindent\textbf{Quantification over a negative variable.}
Now we consider formulas of the following form:
\[
\exists t<0\ \psi(t,\bar x,\bar y).
\]
By the induction hypothesis, \(\psi(t,\bar x,\bar y)\) is equivalent to a finite disjunction of formulas
\[
\chi_i(t,\bar x)
\wedge
\Theta_i(\rho(t),\rho(\bar x),\bar y),
\]
where \(\chi_i(t,\bar x)\) is a pure order formula on \(\mathbb{Q}_{<0}\).

By quantifier elimination for dense linear orders without endpoints, every pure order formula over finitely many parameters is equivalent to a finite disjunction of complete order cells over those parameters. Thus it is enough to treat a formula of the form
\[
\exists t<0\
\left(
\delta(t,\bar x)
\wedge
\Theta(\rho(t),\rho(\bar x),\bar y)
\right),
\]
where \(\delta(t,\bar x)\) is a complete order cell.

There are three cases.

\medskip

\noindent\emph{Case 1: The cell forces \(t=x_j\).}
Assume that \(\delta(t,\bar x)\) implies \(t=x_j\). Then \(\rho(t)=\rho(x_j)\). Hence the formula is equivalent to
\[
\delta'(\bar x)
\wedge
\Theta(\rho(x_j),\rho(\bar x),\bar y),
\]
where \(\delta'(\bar x)\) is the pure order condition obtained by substituting \(x_j\) for \(t\) in \(\delta\). This is in relative normal form.

\medskip

\noindent\emph{Case 2: The cell forces \(t=c\), where \(c\in \mathbb{Q}_{<0}\) is a parameter.}
Assume that \(\delta(t,\bar x)\) implies \(t=c\). Then \(\rho(t)=\rho(c)\). Hence the formula is equivalent to
\[
\delta'(\bar x)
\wedge
\Theta(\rho(c),\rho(\bar x),\bar y),
\]
where \(\rho(c)\) is used as a finite-set parameter in \(\W(\mathbb{Q}_{>0})\). This is again in relative normal form.

\medskip

\noindent\emph{Case 3: The cell lets \(t\) vary in an open interval.}
Assume that \(\delta\) does not force \(t\) to be equal to one of the variables \(x_j\), and does not force \(t\) to be equal to a negative parameter.

For every \(\bar a\in \bigl(\mathbb{Q}_{<0}\bigr)^r\) satisfying the projection of \(\delta\), the set
\[
I_\delta(\bar a):=\{t\in \mathbb{Q}_{<0}:\delta(t,\bar a)\}
\]
is a non-empty open interval of \(\mathbb{Q}_{<0}\), possibly unbounded in \(\mathbb{Q}_{<0}\).

Let
$\pi_\delta(\bar x)$
be the pure order formula
$\exists t<0\ \delta(t,\bar x)$ (which simply asserts that \(I_\delta(\bar x)\) is non-empty).
Since every fiber \(\rho^{-1}(F)\) is dense in \(\mathbb{Q}_{<0}\), for every \(F\in\Fin\) and every \(\bar a\) satisfying \(\pi_\delta\), there exists
$t\in I_\delta(\bar a)$
such that \(\rho(t)=F\). Therefore
\[
\exists t<0\
\left(
\delta(t,\bar x)
\wedge
\Theta(\rho(t),\rho(\bar x),\bar y)
\right)
\]
is equivalent to
\[
\pi_\delta(\bar x)
\wedge
\exists F\in\Fin\ \Theta(F,\rho(\bar x),\bar y),
\]
where \(F\) ranges over the sort $\Fin$. 
The second conjunct is a formula of \(\W(\mathbb{Q}_{>0})\). 
Hence this is again in relative normal form.

This proves that quantification over negative variables preserves relative normal form.

Finally, an unrestricted quantifier is treated by decomposing according to sign:
\[
\exists z\ \psi(z,\bar u)
\]
is equivalent to
\[
\bigl(\exists z<0\ \psi(z,\bar u)\bigr)
\vee
\psi(0,\bar u)
\vee
\bigl(\exists z>0\ \psi(z,\bar u)\bigr).
\]
The cases \(z<0\) and \(z>0\) have already been treated, and the case \(z=0\) is substitution. Thus every formula is equivalent, after sign decomposition, to a formula in relative normal form.
\end{proof}

\section{Open Core}

\subsection{O-minimality of the open core of $(\Q, <,0, A)$ }

\begin{defn}[Open boxes]
Let \(r,s\geq 0\).

An \emph{open box} in \(\bigl(\mathbb{Q}_{<0}\bigr)^r\) is a set of the form
\[
B_{<0}=\prod_{i=1}^r(\ell_i,u_i),
\]
where \(\ell_i,u_i\in \mathbb{Q}_{<0}\) and \(\ell_i<u_i\).

An \emph{open box} in \(\bigl(\mathbb{Q}_{>0}\bigr)^s\) is a set of the form
\[
B_{>0}=\prod_{j=1}^s(p_j,q_j),
\]
where \(p_j,q_j\in \mathbb{Q}_{>0}\) and \(p_j<q_j\).

An open box in \(\bigl(\mathbb{Q}_{<0}\bigr)^r\times \bigl(\mathbb{Q}_{>0}\bigr)^s\) is a product
\[
B_{<0}\times B_{>0}.
\]
If \(r=0\) or \(s=0\), the corresponding product is understood to be a singleton.
\end{defn}

Let \(\delta(\bar x)\) be a complete order cell in \(\bigl(\mathbb{Q}_{<0}\bigr)^r\) over a finite set \(C\subseteq \mathbb{Q}_{<0}\) of parameters.
Call the tuple $(F_1,\dots,F_r)\in\Fin^r$ \textbf{compatible} with (the equalities imposed by) \(\delta\) if:
\begin{itemize}
    \item if \(\delta\) implies \(x_i=x_j\), then \(F_i=F_j\);
    \item if \(\delta\) implies \(x_i=c\) for some \(c\in C\), then \(F_i=\rho(c)\).
\end{itemize}

\begin{lem}
\label{lem:compatible-labels}
Let \(\delta(\bar x)\) be a complete order cell in \(\bigl(\mathbb{Q}_{<0}\bigr)^r\) over a finite set \(C\subseteq \mathbb{Q}_{<0}\) of parameters, and suppose that $(F_1,\dots,F_r)\in\Fin^r$ is compatible with \(\delta\).
Let \(B_{<0}\subseteq \bigl(\mathbb{Q}_{<0}\bigr)^r\) be an open box such that
\[
B_{<0}\cap\delta(\bigl(\mathbb{Q}_{<0}\bigr)^r)\neq\emptyset.
\]
Then there exists
\[
\bar a=(a_1,\dots,a_r)\in B_{<0}\cap\delta(\bigl(\mathbb{Q}_{<0}\bigr)^r)
\]
such that
$\rho(a_i)=F_i$
for every \(i \leq r\).
\end{lem}

\begin{proof}
Choose any $\bar b=(b_1,\dots,b_r)$ from $B_{<0}\cap\delta(\bigl(\mathbb{Q}_{<0}\bigr)^r)$.
The cell \(\delta\) determines an equivalence relation on the variables \(x_1,\dots,x_r\), where \(x_i\) and \(x_j\) are equivalent if \(\delta\) implies \(x_i=x_j\).

Some equivalence classes may be forced to be equal to a parameter \(c\in C\). For such classes, the value of the corresponding coordinates is fixed, and compatibility gives the required label.

For each remaining equivalence class, choose a small open interval $J\subseteq \mathbb{Q}_{<0}$
around the corresponding coordinate of \(\bar b\). We choose these intervals small enough so that:
\begin{itemize}
    \item different free equivalence classes get disjoint intervals;
    \item their relative order is the order prescribed by \(\delta\);
    \item the intervals avoid all parameters from \(C\);
    \item for a class \(E\) of variables, the interval \(J\) is contained in all coordinate intervals of \(B_{<0}\) corresponding to variables in \(E\).
\end{itemize}
This is possible because only finitely many order constraints are involved.

Let \(F\in\Fin\) be the label prescribed for this equivalence class. Since \(\rho^{-1}(F)\) is dense in \(\mathbb{Q}_{<0}\), the interval \(J\) contains some point \(a\) such that $\rho(a)=F$.
Doing this independently for all free equivalence classes gives a tuple
\[
\bar a\in B_{<0}\cap\delta(\bigl(\mathbb{Q}_{<0}\bigr)^r)
\]
with \(\rho(a_i)=F_i\) for every \(i\).
\end{proof}

The following lemma is a special case of a result of Tressl. 
More precisely, Tressl proved that if \(T\) is \(\omega\)-categorical then every model of \(T\) is stably embedded in its weak monadic second-order expansion. 
%citer TRESSL).

\begin{lem}[Stably Embedded]
\label{lem:positive-line}
Let $X\subseteq \bigl(\mathbb{Q}_{>0}\bigr)^n$
be definable in $\W(\mathbb{Q}_{>0})$
with parameters from \(\mathbb{Q}_{>0}\) and from \(\Fin\). Then \(X\) is definable in the pure order \((\mathbb{Q}_{>0},<)\), with parameters from \(\mathbb{Q}_{>0}\).

More precisely, if \(X\) is definable with singleton parameters
$a_1,\dots,a_m\in \mathbb{Q}_{>0}$
and parameters in $\Fin$ given by $F_1,\dots,F_k$,
then \(X\) is definable in \((\mathbb{Q}_{>0},<)\) with parameters from
\[
E:=\{a_1,\dots,a_m\}\cup F_1\cup\cdots\cup F_k.
\]
\end{lem}
\begin{proof}
    This is an immediate consequence of (\cite[Lemma~2.5(ii)]{MT}, since DLO is $\omega$-categorical.
    Since $(\mathbb{Q}_{>0},<)$ is therefore stably embedded in $\W(\mathbb{Q}_{>0})$, any definable subset of $(\mathbb{Q}_{>0})^n$ is also defined by a $<$-formula using the same parameters, by definition.
\end{proof}

\begin{prop}[Interiors are pure-order definable]
\label{prop:interior-pure-order}
Let $X\subseteq \bigl(\mathbb{Q}_{<0}\bigr)^r\times \bigl(\mathbb{Q}_{>0}\bigr)^s$
be definable in $(\Q,<,0,A)$, with parameters. Then the interior of \(X\) (with respect to the product topology on \(\bigl(\mathbb{Q}_{<0}\bigr)^r\times \bigl(\mathbb{Q}_{>0}\bigr)^s\)) is definable in the pure order
$(\mathbb{Q}_{<0},<)\sqcup(\mathbb{Q}_{>0},<)$.

In particular, if \(X\) is open in \(\bigl(\mathbb{Q}_{<0}\bigr)^r\times \bigl(\mathbb{Q}_{>0}\bigr)^s\), then \(X\) is definable in the pure order.
\end{prop}

\begin{proof}
Let \(\varphi(\bar x,\bar y)\) define \(X\), where $\bar x\in \bigl(\mathbb{Q}_{<0}\bigr)^r$ and 
$\bar y\in \bigl(\mathbb{Q}_{>0}\bigr)^s$.
By Proposition~\ref{prop:relative-reduction}, after refining the order conditions into complete order cells, we may write
\[
\varphi(\bar x,\bar y)
\Longleftrightarrow
\bigvee_\delta
\left(
\delta(\bar x)
\wedge
\Theta_\delta(\rho(\bar x),\bar y)
\right),
\]
where the \(\delta\)'s form a finite partition of \(\bigl(\mathbb{Q}_{<0}\bigr)^r\) into complete order cells, and each \(\Theta_\delta\) is a formula of \(\W(\mathbb{Q}_{>0})\). We suppress parameters in the notation; parameters of the form \(\rho(c)\) in the sort $\Fin$ and singleton parameters from \(\mathbb{Q}_{>0}\) are allowed in each of the formulas \(\Theta_\delta\).

Let
\[
B_{<0}=\prod_{i=1}^r(\ell_i,u_i)\subseteq \bigl(\mathbb{Q}_{<0}\bigr)^r
\]
and
\[
B_{>0}=\prod_{j=1}^s(p_j,q_j)\subseteq \bigl(\mathbb{Q}_{>0}\bigr)^s
\]
be open boxes.

For a cell \(\delta\), let $\operatorname{Meet}_\delta(\bar\ell,\bar u)$
denote the $<$-formula stating that $B_{<0}\cap\delta(\bigl(\mathbb{Q}_{<0}\bigr)^r)\neq\emptyset$.
This condition is definable in the pure order on \(\mathbb{Q}_{<0}\).

For finite sets \(F_1,\dots,F_r\in\Fin\), let $\operatorname{Comp}_\delta(F_1,\dots,F_r)$
be the weak monadic formula saying that the tuple of labels is compatible with the equalities imposed by \(\delta\). 
Explicitly, it says that $F_i=F_j$ whenever \(\delta\) implies \(x_i=x_j\), and that $F_i=\rho(c)$
whenever \(\delta\) implies \(x_i=c\) for some parameter $c<0$.

We now define in \(\W(\mathbb{Q}_{>0})\) the formula $\Gamma_\delta(\bar p,\bar q)$
as follows:
\[
\begin{aligned}
\forall F_1,\dots,F_r\in\Fin\ \forall z_1,\dots,z_s\in \mathbb{Q}_{>0}\ \bigl[&
\operatorname{Comp}_\delta(F_1,\dots,F_r)
\wedge
\bigwedge_{j=1}^s p_j<z_j<q_j\\
&\Rightarrow
\Theta_\delta(F_1,\dots,F_r,z_1,\dots,z_s)
\bigr].
\end{aligned}
\]

Observe that this formula describes the boxes $B_{\bar{p},\bar{q}}$ with endpoints coming from $\bar{p}$ and $\bar{q}$ (on the left and right respectively) on which the tuple $(F_1,\dots,F_r)$ being compatible guarantees that $\Theta_{\delta}$ is satisfied by said tuple in conjunction with any element $\bar{z}$ of the box $B_{\bar{p},\bar{q}}$.
This is a formula of \(\W(\mathbb{Q}_{>0})\) whose free variables are elements of \(\mathbb{Q}_{>0}\). By Lemma~\ref{lem:positive-line}, it is equivalent to a pure order formula on \(\mathbb{Q}_{>0}\); let $\gamma_\delta(\bar p,\bar q)$ be such a $<$-formula.

We claim that
\[
B_{<0}\times B_{>0}\subseteq X
\]
if and only if, for every cell \(\delta\),
\[
\operatorname{Meet}_\delta(\bar\ell,\bar u)
\Rightarrow
\gamma_\delta(\bar p,\bar q).
\]

Assume first that \(B_{<0}\times B_{>0}\subseteq X\). Let \(\delta\) be a cell such that
\[
B_{<0}\cap\delta(\bigl(\mathbb{Q}_{<0}\bigr)^r)\neq\emptyset.
\]
Let $(F_1,\dots,F_r)\in\Fin^r$ be compatible with \(\delta\). 
By Lemma~\ref{lem:compatible-labels}, there exists
\[
\bar a\in B_{<0}\cap\delta(\bigl(\mathbb{Q}_{<0}\bigr)^r)
\]
such that \(\rho(a_i)=F_i\) for every \(i\). 
Since \(B_{<0}\times B_{>0}\subseteq X\), for every \(\bar z\in B_{>0}\) we have $(\bar a,\bar z)\in X$.
Using the normal form on the cell \(\delta\), this yields $\Theta_\delta(F_1,\dots,F_r,\bar z)$.
Thus \(\Gamma_\delta(\bar p,\bar q)\) holds, and therefore \(\gamma_\delta(\bar p,\bar q)\) holds.

Conversely, assume that for every cell \(\delta\),
\[
\operatorname{Meet}_\delta(\bar\ell,\bar u)
\Rightarrow
\gamma_\delta(\bar p,\bar q).
\]
We can translate this as saying that if $\delta$ intersects $B_{<0}$, then for every tuple $\bar{F}$ of compatible parameters coming from $\Fin$ and for every element $\bar{z}$ of $B_{>0}$, the pair $(\bar{F}, \bar{z})$ satisfy $\Theta_{\delta}$.
Suppose that $\bar a\in B_{<0}$ and $\bar z\in B_{>0}$.
There is a unique cell \(\delta\) such that \(\delta(\bar a)\) holds. 
This $\bar{a}$ then witnesses that 
\(B_{<0}\cap\delta(\bigl(\mathbb{Q}_{<0}\bigr)^r)\neq\emptyset\), i.\ e., $\bar{a}$ witnesses that $\operatorname{Meet}_\delta(\bar\ell,\bar u)$ holds.
By hypothesis, it follows that \(\gamma_\delta(\bar p,\bar q)\) holds,
so \(\Gamma_\delta(\bar p,\bar q)\) holds as well.

Due to how we partition $\varphi$ using $\delta$, the tuple $(\rho(a_1),\dots,\rho(a_r))$
is necessarily compatible with \(\delta\). Since \(\bar z\in B_{>0}\), by $\Gamma_\delta(\bar p,\bar q)$ we get
\[
\Theta_\delta(\rho(a_1),\dots,\rho(a_r),\bar z).
\]
Together with \(\delta(\bar a)\), the normal form yields $(\bar a,\bar z)\in X$.
Thus \(B_{<0}\times B_{>0}\subseteq X\).

Therefore the relation
\[
B_{<0}\times B_{>0}\subseteq X
\]
is definable in the pure order.

Hence an element \((\bar a,\bar b)\in \bigl(\mathbb{Q}_{<0}\bigr)^r\times \bigl(\mathbb{Q}_{>0}\bigr)^s\) belongs to \(\Int(X)\) if and only if there exist endpoints
$\ell_i<a_i<u_i$ for $i=1,\dots,r$
in \(\mathbb{Q}_{<0}\) and
$p_j<b_j<q_j$ for $j=1,\dots,s$
in \(\mathbb{Q}_{>0}\) such that
\[
\prod_i(\ell_i,u_i)\times\prod_j(p_j,q_j)\subseteq X.
\]
We've shown that this is expressible in the pure order, thus \(\Int(X)\) is definable in the pure order.
If \(X\) is open, then \(X=\Int(X)\), so \(X\) is pure-order definable.
\end{proof}

\begin{cor}
\label{cor:open-sets-qn}
Every open subset of \(\mathbb{Q}^n\) definable in $(\Q,<,0,A)$ with parameters is definable in the pure order \((\mathbb{Q},<)\) (with the same parameters).
\end{cor}

\begin{proof}
Let $U\subseteq\mathbb{Q}^n$ be open and definable in $(\Q,<,0,A)$.

We decompose \(\mathbb{Q}^n\) into finitely many sign strata. 
A sign stratum is given by choosing, for each coordinate \(z_i\), one of the following:
\[
z_i<0,\qquad z_i=0,\qquad z_i>0.
\]
Let \(S\) be such a sign stratum. 
After deleting the coordinates with fixed value \(0\), the sign stratum \(S\) is naturally identified with a product $\bigl(\mathbb{Q}_{<0}\bigr)^r\times \bigl(\mathbb{Q}_{>0}\bigr)^s$
where $r+s\leq n$.

By Proposition~\ref{prop:interior-pure-order}, the set \(U\cap S \subseteq \Q^n\) is pure-order definable. 
We now observe that $\bigcup_S(U\cap S)$ is a finite union, and we may write $U\setminus \bigcup_S(U\cap S)$ as follows:
$$\{P_i:=U \cap (x_1, \ldots ,x_{i-1},0,x_{i+1}, \ldots ,x_n): i \leq n\}.$$
Since $U$ is open, so is each $P_i$ in $\Q^{n-1}$.

Repeating this argument for each $P_i$, we see that we can write $U$ as a finite union of the images of sign strata under the map $\iota_{J}:S \hookrightarrow \Q^n$ where $J \subseteq \{1, \ldots ,n\}$ specifies which coordinates of $\iota_{J}(S)$ have constant value $0$, and for each $\bar{x} \in S$, $\iota$ maps $x_i$ to the $i^{th}$ coordinate within the set $\{1, \ldots ,n\} \setminus J$ (i.\ e., the most natural way to embed $S$ ``around'' the fixed value 0 coordinates). 
Clearly $\iota_{J}(S)$ is also pure-order definable for each sign stratum in every arity $i \leq n$, so $U$ is also pure-order definable, as desired.
\end{proof}

\begin{thm}
The open core of $(\Q,<,0,A)$ is o-minimal.
\end{thm}

\subsection{Main Result}
\begin{rmk}
Suppose that $\calR \succcurlyeq (\Q,<,0,A)$ is an $\aleph_1$-saturated elementary extension with infinite elements (relative to the elements of $\Q$).
Then $A^{\calR}$ contains infinite discrete fibers.
\begin{proof}
This follows from the standard compactness argument.
\end{proof}
\end{rmk}

\begin{cor}
The open core of $\calR$ is not o-minimal.
\begin{proof}
The open core of $\calR$ includes the topological closure of each fiber $A_{x}^{\calR}$ with $x \in \calR$, since these are definable.
Hence $\calR \setminus \overline{A_{x}^{\calR}}$ is definable for each $x \in \calR$.
Yet for at least one $x \in \calR$, we know $\calR \setminus \overline{A_{x}^{\calR}}$ is an infinite union of open intervals that cannot be expressed as a finite union of open intervals.
\end{proof}
\end{cor}

\begin{cor}[Theorem \ref{Main}]
The property of having o-minimal open core is not elementary.
\end{cor}

\section*{Acknowledgments}
The first named author was supported by Cofund Math In Greater Paris, Marie Sk\l odowska-Curie Actions (H2020-MSCA-COFUND-2020-GA101034255).
Many thanks to Philipp Hieronymi for his insightful comments and excellent suggestions, as well as the anonymous referee for their helpful comments and questions as well.

\end{document}